\documentclass{amsart}

\usepackage[margin=1in]{geometry}
\usepackage{mathrsfs}
\usepackage{tikz-cd}
\usepackage[pdfusetitle]{hyperref}
\usepackage[alphabetic,backrefs]{amsrefs}

\newtheorem{theorem}{Theorem}
\newtheorem{maintheorem}{Theorem}

\newtheorem{lemma}[theorem]{Lemma}
\newtheorem{corollary}[theorem]{Corollary}
\newtheorem{maincorollary}[maintheorem]{Corollary}

\theoremstyle{remark}
\newtheorem{example}[theorem]{Example}

\theoremstyle{definition}
\newtheorem{definition}[theorem]{Definition}

\numberwithin{theorem}{section}

\hypersetup{
  colorlinks=true,
  linkcolor=blue, 
  urlcolor=blue,
  citecolor=blue
}

\begin{document}
  \title{
    Contractibility of the geometric stability manifold of a surface
  }
  \author{
    Nick Rekuski
  }
  \address{
    Department of Mathematical Sciences, University of Liverpool, Liverpool, L69 7ZL,
United Kingdom
  }
  \email{
    rekuski@liverpool.ac.uk
  }
  \date{
    May 23, 2024
  }
  \thanks{
    The author was supported by the U.S. Department of Energy, Office of Science, Basic Energy Sciences, under Award Number DE-SC-SC0022134 and an OVPR Postdoctoral Award at Wayne State University.
  }
  \subjclass[2020]{Primary 14F08, 14J60, 18G80}
  \keywords{Bridgeland stability, geometric stability conditions, contractibility, homotopy type}

  \begin{abstract}
    Using a recent description of the geometric stability manifold, we show the geometric stability manifold associated to any smooth projective complex surface is contractible.
    We then use this result to demonstrate infinitely many new families of surfaces whose stability manifold is contractible.
  \end{abstract}

  \maketitle

\section{Introduction}
   On a triangulated category, $\mathcal{T}$, Bridgeland defined the notion of a stability condition which establishes a framework to study moduli of complexes of sheaves~\cite{bridgeland2007}.
   The collection of all stability conditions forms a complex manifold (with possibly infinitely many connected components), $\operatorname{Stab}(\mathcal{T})$.
   Stability conditions and the stability manifold have been used to obtain new results in homological mirror symmetry, representation theory, symplectic geometry, and moduli of stable sheaves.
   Furthermore, $\operatorname{Stab}(\mathcal{T})$ has an inherent richness that is worthy of study in its own right.
   
  Recent work has focused on understanding the topology of $\operatorname{Stab}(\mathcal{T})$.
  This focus is partially inspired by a conjecture of Bridgeland that states if $X$ is a $K3$ surface then a distinguished component of $\operatorname{Stab}(X)$---the stability manifold of $D^{b}(X)$---is simply-connected~\cite{bridgeland2008}*{Conjecture 1.2}.
  A positive answer to this conjecture would shed light on the group of exact autoequivalences of the bounded derived category of a $K3$ surface.
  More generally, there is a folklore conjecture that if $\operatorname{Stab}(\mathcal{T})$ is nonempty then it is homotopy discrete.
  
  In this article we consider the case when $\mathcal{T}=D^{b}(X)$ for $X$ a smooth projective complex variety.
  In practice, contractibility results for $\operatorname{Stab}(X)$ first give a homotopy equivalence to a certain natural submanifold, $\operatorname{Stab}^{\mathrm{geo}}(X)$ (called the geometric stability manifold), then show $\operatorname{Stab}^{\mathrm{geo}}(X)$ is contractible (we give an overview of such results in the following subsection).
  Both steps of this procedure depend on the geometry of $X$, so with current techniques it is intractable to show $\operatorname{Stab}(X)$ is contractible in general.
  
  The main result of this article shows, in the case of surfaces, the second step of this procedure can be performed in general.
  
    \begin{maintheorem}[\ref{theorem:mainTheoremNumbered}]
    \label{thm:mainTheorem}
    If $X$ is a smooth projective complex surface then $\operatorname{Stab}^{\mathrm{geo}}(X)$ is contractible.
  \end{maintheorem}
  
  Using Theorem \ref{thm:mainTheorem} and a recent result of Fu--Li--Zhao, we produce infinitely many new surfaces with contractible stability manifold.
  
  \begin{maincorollary}[\ref{corollary:finiteAlbanese}]
    Suppose $X$ is a smooth projective complex surface.
    If the Albanese morphism of $X$ is finite onto its image then $\operatorname{Stab}(X)$ is contractible.
  \end{maincorollary}

\subsection*{History}
  We give a historical overview of topological results for $\operatorname{Stab}(\mathcal{T})$.
  It is intractable to show $\operatorname{Stab}(\mathcal{T})$ is contractible directly.
  Instead, results typically focus on two natural submanifolds of $\operatorname{Stab}(\mathcal{T})$: the algebraic stability manifold, $\operatorname{Stab}^{\mathrm{alg}}(\mathcal{T})$, and the geometric stability manifold, $\operatorname{Stab}^{\mathrm{geo}}(\mathcal{T})$.
  The algebraic stability manifold consists of stability conditions arising from full exceptional collections of $\mathcal{T}$, and the geometric stability manifold consists of stability conditions where skyscraper sheaves are stable.
  We note some recent results have also considered a third submanifold consisting of Serre invariant stability conditions~\cites{pertusi2022,fan2023}.

  In many cases $\operatorname{Stab}^{\mathrm{geo}}(\mathcal{T})\cup\operatorname{Stab}^{\mathrm{alg}}(\mathcal{T})$ is a connected component of $\operatorname{Stab}(\mathcal{T})$.
  Therefore, the homotopy types of $\operatorname{Stab}^{\mathrm{alg}}(\mathcal{T})$ and $\operatorname{Stab}^{\mathrm{geo}}(\mathcal{T})$ provide important topological information about $\operatorname{Stab}(\mathcal{T})$.
  For example, if $\mathcal{T}=D^{b}(Q)$---the bounded derived category of representations of a quiver $Q$--and the underlying graph of $Q$ is simple and acyclic then $\operatorname{Stab}^{\mathrm{alg}}(Q)=\operatorname{Stab}(Q)$~\cite{otani2022}*{Theorem 1.1}.
  Similarly, if $\mathcal{T}=D^{b}(X)$ and $X$ has finite Albanese then $\operatorname{Stab}^{\mathrm{geo}}(X)=\operatorname{Stab}(X)$~\cite{fu2022}*{Theorem 1.1}.

  We give a brief overview of topological results involving $\operatorname{Stab}^{\mathrm{alg}}(\mathcal{T})$ and $\operatorname{Stab}^{\mathrm{geo}}(\mathcal{T})$.
  We first describe results for $\operatorname{Stab}^{\mathrm{geo}}(\mathcal{T})$.
  \begin{itemize}
    \item{
      If $\mathcal{T}=D^{b}(X)$ where $X$ is a $K3$ surface then $\operatorname{Stab}^{\mathrm{geo}}(\mathcal{T})$ is contractible~\cite{bridgeland2008}.
    }
    \item{
      If $\mathcal{T}=D^{b}(X)$ where $X$ is a smooth positive genus curve then $\operatorname{Stab}^{\mathrm{geo}}(\mathcal{T})=\operatorname{Stab}(\mathcal{T})$ is contractible~\cite{macri2007},
    }
    \item{
      If $\mathcal{T}=D^{b}(X)$ for an abelian surface $X$ then $\operatorname{Stab}^{\mathrm{geo}}(\mathcal{T})=\operatorname{Stab}(\mathcal{T})$ is simply connected and if the N{\'e}ron-Severi rank of $X$ is $1$ then $\operatorname{Stab}(\mathcal{T})$ is contractible~\cites{huybrechts2008,fu2022}.
    }
    \item{
      If $\mathcal{T}=D^{b}(X)$ where $X$ is an irregular surface of N{\'e}ron-Severi rank one then $\operatorname{Stab}^{\mathrm{geo}}(\mathcal{T})=\operatorname{Stab}(\mathcal{T})$ is contractible~\cite{fu2022}.
    }
    \item{
      If $\mathcal{T}=D^{b}(X)$ where $X$ is an abelian threefold of N{\'e}ron-Severi rank one then $\operatorname{Stab}^{\mathrm{geo}}(\mathcal{T})=\operatorname{Stab}(\mathcal{T})$ is contractible~\cite{fu2022}.
    }
  \end{itemize}
  
  We now discuss $\operatorname{Stab}^{\mathrm{alg}}(\mathcal{T})$ (see Figure \ref{figure:quiverTypes} for illustrations of the relevant quivers).
  \begin{itemize}
    \item{
      If $\mathcal{T}=D^{b}(Q)$ where $Q$ is a quiver of type $A_{n}$, $D_{n}$, or $E_{n}$ then $\operatorname{Stab}^{\mathrm{alg}}(\mathcal{T})=\operatorname{Stab}(\mathcal{T})$ is contractible~\cites{haiden2014,qiu2015,qiu2018,bridgeland2020}.
    }
    \item{
      If $\mathcal{T}=D^{b}(Q)$ where $Q$ is of type $\widetilde{A}_{2}$ then $\operatorname{Stab}^{\mathrm{alg}}(\mathcal{T})=\operatorname{Stab}(\mathcal{T})$ is contractible~\cites{dimitrov2016,dimitrov2018}.
    }
    \item{
      If $\mathcal{T}=D^{b}(K_{n})$ where $K_{n}$ is the Konecker quiver (including the case $\mathcal{T}=D^{b}(\mathbb{P}^{1})$) then $\operatorname{Stab}^{\mathrm{alg}}(\mathcal{T})=\operatorname{Stab}(\mathcal{T})$ is contractible~\cites{okada2006,macri2007,dimitrov2019}.
    }
    \item{
      If $\mathcal{T}=D^{b}(S_{3})$ (equivalently $\mathcal{T}=D^{b}(\mathbb{P}^{2})$) then $\operatorname{Stab}^{\mathrm{alg}}(\mathcal{T})$ is contractible~\cites{macri2007b,li2017}.
    }
    \item{
      If $\mathcal{T}=D^{b}(S_{4})$ (equivalently $\mathcal{T}=D^{b}(\mathbb{P}^{3})$) then $\operatorname{Stab}^{\mathrm{alg}}(\mathcal{T})$ is simply connected~\cite{chen2020}.
    }
    \item{
      If $\mathcal{T}=D^{b}(Q_{r,n,m})$ then $\operatorname{Stab}^{\mathrm{alg}}(\mathcal{T})=\operatorname{Stab}(\mathcal{T})$ is contractible~\cite{broomhead2016}.
    }
  \end{itemize}
  
  Some cases that do not strictly fit into the groupings above include:
  \begin{itemize}
    \item{
      If $\mathcal{T}=D^{b}(\mathbb{P}^{2})$ then there is a contractible connected component of $\operatorname{Stab}(\mathcal{T})$, $\operatorname{Stab}^{\dagger}(\mathcal{T})$, with $\operatorname{Stab}^{\dagger}(\mathcal{T})=\operatorname{Stab}^{\mathrm{geo}}(\mathcal{T})\cup\operatorname{Stab}^{\mathrm{geo}}(\mathcal{T})$~\cite{li2017}.
    }
    \item{
      If $\mathcal{T}=D^{b}(X)$ where $X$ is a $K3$ surface of Picard rank one then the connected component of $\operatorname{Stab}(\mathcal{T})$ containing $\operatorname{Stab}^{\mathrm{geo}}(\mathcal{T})$ is contractible~\cite{bayer2017}.
    }
    \item{
      If $\mathcal{T}=D^{b}(X)$ where $X$ is a $K3$ surface of Picard rank zero (in particular, $X$ is not projective) then a $\operatorname{Stab}(\mathcal{T})$ is simply connected~\cite{huybrechts2008}.
    }
    \item{
      If $\mathcal{T}=D^{b}_{\mathbb{P}^{2}}(\operatorname{Tot}(\mathscr{O}_{\mathbb{P}^{2}}(-3)))$ then a distinguished connected component of $\operatorname{Stab}(\mathcal{T})$ is simply connected~\cite{bayer2011}.
    }
  \end{itemize}
  
\begin{figure}
  \label{figure:quiverTypes}
  \begin{center}
    $A_{n}$:
    \begin{tikzcd}
      1\arrow[r,dash] & 2\arrow[r,dash] & \cdots\arrow[r,dash] & n
    \end{tikzcd}
    \qquad
    $K_{n}$:
    \begin{tikzcd}
      1\arrow[rr,bend left=60,"1"]\arrow[rr,bend left=20,"2"]\arrow[rr,bend right=10,draw=none,"\vdots" description]\arrow[rr,bend right=60,swap,"n"] & & 2
    \end{tikzcd}
    
    $D_{n}$:
    \begin{tikzcd}
      1\arrow[r,dash] & 2\arrow[r,dash]\arrow[d,dash] & 3\arrow[r,dash] &\cdots \arrow[r,dash] & n-1 \\
       & n & & &
    \end{tikzcd}
    \qquad
    $\widetilde{A}_{2}$:
    \begin{tikzcd}
      1\arrow[rr,dash]\arrow[dr,dash] & & 2\arrow[dl,dash]\\
       & 3 &
    \end{tikzcd}
    
    $E_{n}$ ($n=6,7,8$):
    \begin{tikzcd}
      1\arrow[r,dash] & 2\arrow[r,dash] & 3\arrow[r,dash]\arrow[d,dash] & 4\arrow[r,dash] &\cdots\arrow[r,dash] & n-1 \\
       & & n & & & 
    \end{tikzcd}
    
    $S_{n}$:
    \begin{tikzcd}
      1\arrow[rr,bend left=60,"\phi_{1}^{1}"]\arrow[rr,bend left=20,"\phi_{1}^{2}"]\arrow[rr,bend right=60,swap,"\phi_{1}^{n}"]\arrow[rr,bend right=10,draw=none,"\vdots" description] & & 2\arrow[rr,bend left=60,"\phi_{2}^{1}",shorten >=3ex]\arrow[rr,bend left=20,"\phi_{2}^{2}",shorten >=3ex]\arrow[rr,bend right=10,draw=none,"\vdots" description,shorten >=3ex]\arrow[rr,bend right=60,swap,"\phi_{2}^{n}",shorten >=3ex] & & \cdots\arrow[rr,bend left=60,"\phi_{n-1}^{1}",shorten <=3ex]\arrow[rr,bend left=20,"\phi_{n-1}^{2}",shorten <=3ex]\arrow[rr,bend right=60,swap,"\phi_{n-1}^{n}",shorten <=3ex]\arrow[rr,bend right=10,draw=none,"\vdots" description,shorten <=3ex] & & n
    \end{tikzcd}
    with $\phi_{i+1}^{j}\phi_{i}^{k}=\phi_{i+1}^{k}\phi_{i}^{j}$.
    
    $Q_{r,n,m}$:
    \begin{tikzcd}
       &  & &  & 2\arrow[dr,bend left=45,shorten >=2ex,"\gamma_{2}"] &\\
      1\arrow[r] & 2\arrow[r] & \cdots\arrow[r] & m\arrow[ur,bend left=45,"\gamma_{1}"] & & \vdots\arrow[dl,bend left=45,shorten <=2ex,"\gamma_{n-1}"] \\
      &  & &  & n\arrow[ul,bend left=45,"\gamma_{n}"] &\\
    \end{tikzcd}
    with $\gamma_{n-i+1}\gamma_{n-i+2}=0$ for all $0\leq i\leq r$.
    
    \caption{
      Illustration of relevant quivers.
    }
  \end{center}
  \end{figure}
 
\subsection*{Notation and Assumptions}
  Suppose $X$ is a smooth projective complex surface.
  We write $\operatorname{NS}(X)$ for the N{\'e}ron-Severi group of $X$ and $\operatorname{Amp}(X)$ for the ample cone.
  We also write $\operatorname{NS}_{\mathbb{R}}(X)=\operatorname{NS}(X)\otimes\mathbb{R}$ and $\operatorname{Amp}_{\mathbb{R}}(X)=\operatorname{Amp}(X)\otimes\mathbb{R}$.
  We use greek letters (e.g. $\alpha,\beta,\phi$) for real numbers.
  
  The bounded derived category of $\operatorname{Coh}(X)$ will be written $D^{b}(X)$ and the numerical Grothendieck group of $X$ will be written $K_{\operatorname{num}}(X)$.
  We use capital script letters (e.g. $\mathscr{E}$) for coherent sheaves on $X$.
  We use capital letters (e.g. $E$, $F$) for chain complexes in $D^{b}(X)$.
  We write $\operatorname{ch}_{i}(\mathscr{E})$ for the $i$th Chern character of $\mathscr{E}$ viewed as an element of $H^{2i}(X,\mathbb{Q})$.
  
\subsection*{Acknowledgments}
  The author is thankful to Hannah Dell, Rajesh Kulkarni, and Andrew Salch for discussions related to this work.

\section{Preliminaries}
\label{section:preliminaries}
  We recall background on stability conditions.
  
  \begin{definition}
    A \emph{stability condition} on $D^{b}(X)$ (also called a \emph{Bridgeland stability condition}) is a pair $\sigma=(Z,\mathcal{P})$ where $Z:K_{\operatorname{num}}(X)\to\mathbb{C}$ is a group homomorphism, called a \emph{central charge}, and $\{\mathcal{P}(\phi)\}_{\phi\in\mathbb{R}}$ is a collection of full subcategories of $D^{b}(X)$, called a \emph{slicing}, satisfying the following conditions
    \begin{itemize}
      \item{
        If $E\in\mathcal{P}(\phi)\setminus\{0\}$ then $Z([E])\in\mathbb{R}_{>0}\exp(i\pi\phi)$---the positive real ray spanned by $\exp(i\pi\phi)$,
      }
      \item{
        If $\phi\in\mathbb{R}$ then $\mathcal{P}(\phi+1)=\mathcal{P}(\phi)[1]$,
      }
      \item{
        If $\phi_{1}>\phi_{2}$ then $\operatorname{Hom}_{D^{b}(X)}(\mathcal{P}(\phi_{1}),\mathcal{P}(\phi_{2}))=0$,
      }
      \item{
        If $E\in D^{b}(X)\setminus\{0\}$ then there exists real numbers $\phi_{1}>\phi_{2}>\cdots>\phi_{n}$ and a sequence of distinguished triangles
        \[
          \begin{tikzcd}
            0\arrow[r] & E_{1}\arrow[r]\arrow[d] & E_{2}\arrow[r]\arrow[d] & \cdots\arrow[r] & E_{n-1}\arrow[r]\arrow[d] & E\arrow[d] \\
             & F_{1}\arrow[ul,dashed] & F_{2}\arrow[ul,dashed] & \cdots & F_{n-1}\arrow[ul,dashed] & F_{n}\arrow[ul,dashed]
          \end{tikzcd}
        \]
        where $F_{i}\in\mathcal{P}(\phi_{i})$, and
      }
      \item{
        $\displaystyle \inf_{\phi\in\mathbb{R}}\left\{\frac{|Z([E])|}{||[E]||}:E\in\mathcal{P}(\phi)\setminus\{0\}\right\}>0$ for some norm $||\cdot||$ on $K_{\mathrm{num}}(X)$.
      }
    \end{itemize}
    
    If $\sigma$ is a stability condition and $E\in\mathcal{P}(\phi)\setminus\{0\}$ then we say $E$ is \emph{$\sigma$-semistable of phase $\phi$}.
    If $E\in\mathcal{P}(\phi)\setminus\{0\}$ is simple in $\mathcal{P}(\phi)$ then we say $E$ is \emph{$\sigma$-stable of phase $\phi$}.
  \end{definition}
  
  We write the set of all stability conditions on $D^{b}(X)$ as $\operatorname{Stab}(X)$.
  Bridgeland described a natural metric on $\operatorname{Stab}(X)$ so that $\operatorname{Stab}(X)$ is a complex manifold (with possibly infinitely many connected components)~\cite{bridgeland2007}*{Theorem 1.2}.
  For our purposes, we restrict our attention to an open submanifold of $\operatorname{Stab}(X)$ where all skyscraper sheaves are stable.
  
  \begin{definition}
    \label{definition:geometricStabilityCondition}
    A stability condition on $\sigma\in\operatorname{Stab}(X)$ is said to be \emph{geometric} if the skyscraper sheaf $\mathscr{O}_{x}$ is $\sigma$-stable for all $x\in X$.
    We write the collection of all geometric stability conditions on $D^{b}(X)$ as $\operatorname{Stab}^{\mathrm{geo}}(X)$.
  \end{definition}
  
  Commonly, geometric stability conditions are defined to be stability conditions where skyscraper sheaves are all stable of the same phase.
  We have dropped this additional assumption from Definition \ref{definition:geometricStabilityCondition} because~\cite{fu2022}*{Proposition 2.9} has shown it is already implied by the stability assumption.
  
  The subspace $\operatorname{Stab}^{\mathrm{geo}}(X)$ is open in $\operatorname{Stab}(X)$, so each connected component of $\operatorname{Stab}^{\mathrm{geo}}(X)$ is a complex submanifold of some connected component of $\operatorname{Stab}(X)$~ \cite{bridgeland2008}*{Proposition 9.4}.
  
  The explicit description of $\operatorname{Stab}^{\mathrm{geo}}(X)$ relies on a generalization of the Le Potier function.
  We first recall this function then the description of $\operatorname{Stab}^{\mathrm{geo}}(X)$.
  
  \begin{definition}[\cite{fu2022}*{Definition 3.1},\cite{dell2023}*{Definition 5.8}]
    \label{definition:generalizedLePotier}
    Suppose $X$ is a smooth projective complex surface.
    We define the \emph{Le Potier function} $\Phi:\operatorname{Amp}_{\mathbb{R}}(X)\times\operatorname{NS}_{\mathbb{R}}(X)\times\mathbb{R}\to\mathbb{R}$ as
    \[
      \Phi(H,D,\beta)=\limsup_{\mu\to\beta}\left\{\frac{\operatorname{ch}_{2}(\mathscr{E})-D\cdot\operatorname{ch}_{1}(\mathscr{E})}{H^{2}\operatorname{rank}(\mathscr{E})}:
      \begin{array}{l}
        \mathscr{E}\in\operatorname{Coh}(X) \ \text{is slope stable (with respect} \\
        \text{to H) and} \ (H\cdot\operatorname{ch}_{1}(\mathscr{E}))/\operatorname{rank}(\mathscr{E})=\mu
      \end{array}{}\right\}.
    \]
  \end{definition}
  
  We now give the description of $\operatorname{Stab}^{\mathrm{geo}}(X)$.
  
  \begin{theorem}[\cite{fu2022}*{Proposition 3.6}, \cite{dell2023}*{Theorem 5.10}]
    \label{contractabilityMainResult}
    Suppose $X$ is a smooth projective complex surface.
    There is a homeomorphism
    \[
      \operatorname{Stab}^{\mathrm{geo}}(X)\cong\mathbb{C}\times\left\{(H,D,\beta,\alpha)\in\operatorname{Amp}_{\mathbb{R}}(X)\times\operatorname{NS}_{\mathbb{R}}(X)\times\mathbb{R}\times\mathbb{R}:\Phi(H,D,\beta)<\alpha\right\}.
    \]
  \end{theorem}
  
  We broadly sketch the proof of this theorem.
  In the case of surfaces, geometric stability conditions are determined by their central charge~\cite{bridgeland2008}*{Proposition 10.3}.
  By considering the central charge of a skyscraper sheaf and ideal sheaves of curves there is a continuous injection $i:\operatorname{Stab}^{\mathrm{geo}}(X)\to\mathbb{C}\times\operatorname{Amp}_{\mathbb{R}}(X)\times\operatorname{NS}_{\mathbb{R}}(X)\times\mathbb{R}\times\mathbb{R}$~\cite{bridgeland2008}*{Section 10}, \cite{fu2022}*{Proposition 3.6},\cite{dell2023}*{Theorem 5.5}.
  By directly constructing a quadratic form, one shows the support property is satisfied by exactly $(\lambda,H,D,\beta,\alpha)\in i(\operatorname{Stab}^{\mathrm{geo}}(X))$ satisfying $\Phi(H,D,\beta)<\alpha$~\cite{dell2023}*{Lemma 5.32}.
  
\section{Contractibility of the Geometric Stability Manifold}
  Contractibility of $\operatorname{Stab}^{\mathrm{geo}}(X)$ will follow from Theorem \ref{contractabilityMainResult} and the following topological result.
  This lemma shows the region above the graph of a ``sufficiently nice" function is contractible.
  Note we cannot assume this function is continuous because the Le Potier function is not continuous in general!
  
  \begin{lemma}
    \label{lemma:contractibleContinuous}
    Suppose $Z$ is a topological space and $f:Z\to\mathbb{R}$ is a function (not necessarily continuous).
    Define
    \[
      \Gamma_{f}^{<}=\{(z,\alpha)\in Z\times\mathbb{R}\mid f(z)<\alpha\}
    \]
    which we give the subspace topology.
    If there exists a continuous function $g:Z\to\mathbb{R}$ satisfying $f(z)\leq g(z)$ for all $z\in Z$ then $\Gamma_{f}^{<}$ and $Z$ are homotopy equivalent.
  \end{lemma}
  \begin{proof}
    By replacing $g(z)$ with $g(z)+1$ we may assume $f(z)<g(z)$ (this assumption makes notation simpler).
    Since $f(z)<g(z)$ for all $z\in Z$, the function $F:\Gamma_{f}^{<}\times [0,1]\to\Gamma_{f}^{<}$ defined by
    \[
      F((z,\alpha),t)=(z,\max\{\alpha,(g(z)-\alpha)t+\alpha\})
    \]
    is well-defined and continuous.
    Moreover, since $\max\{\alpha,g(z)\}\geq g(z)$, $F$ is a strong deformation retract of $\Gamma_{f}^{<}$ onto
    \[
      \Gamma_{g}^{\leq}=\{(z,\beta)\in Z\times\mathbb{R}\mid g(z)\leq\beta\}.
    \]
    Furthermore, the function $G:\Gamma_{g}^{\leq}\times [0,1]\to\Gamma_{g}^{\leq}$ defined by
    \[
      G((z,\beta),t)=\left(z,\beta(1-t)+g(z)t\right)
    \]
    is a strong deformation retract of $\Gamma_{g}^{\leq}$ onto the graph of $g$.
    Since $g$ is continuous, the graph of $g$ is homeomorphic to $Z$.
    Hence, we have shown $\Gamma_{f}^{<}$ is homotopy equivalent to $Z$, as desired.
  \end{proof}
  
  Lemma \ref{lemma:contractibleContinuous} does not generalize to arbitrary functions $f:Z\to\mathbb{R}$.
  For example, if $Z=\mathbb{R}$ and $f(z)=1/z$ with $f(0)=0$ then $\Gamma_{f}^{<}$ has multiple path components while $Z$ is path-connected.
  Either way, this is not an issue in our case.

  \begin{theorem}
    \label{theorem:mainTheoremNumbered}
    Suppose $X$ is a smooth projective surface over $\mathbb{C}$.
    The manifold $\operatorname{Stab}^{\mathrm{geo}}(X)$ is contractible.
  \end{theorem}
  \begin{proof}
    By \cite{dell2023}*{Lemma 4.6}, we have
    \[
      \Phi(H,D,\beta)\leq \frac{1}{2}\left(\left(\beta-\frac{D\cdot H}{H^{2}}\right)^{2}-\frac{D^{2}}{H^{2}}\right)
    \]
    for all $(H,D,\beta)\in\operatorname{Amp}_{\mathbb{R}}(X)\times\operatorname{NS}_{\mathbb{R}}(X)\times\mathbb{R}$ where $\Phi$ is the generalized Le Potier function (see Definition \ref{definition:generalizedLePotier}).
    Therefore, by Theorem \ref{contractabilityMainResult} and Lemma \ref{lemma:contractibleContinuous}, it suffices to show $\operatorname{Amp}_{\mathbb{R}}(X)\times\operatorname{NS}_{\mathbb{R}}(X)\times\mathbb{R}$ is contractible.
    
    Since $\operatorname{NS}(X)$ is a finitely generated abelian group, there is a homeomorphism $\operatorname{NS}_{\mathbb{R}}(X)=\mathbb{R}^{m}$ for some $m\geq 1$.
    Similarly, since $\operatorname{Amp}_{\mathbb{R}}(X)$ is an open cone in $\operatorname{NS}_{\mathbb{R}}(X)=\mathbb{R}^{m}$, $\operatorname{Amp}_{\mathbb{R}}(X)$ is also contractible.
    Thus, $\operatorname{NS}_{\mathbb{R}}(X)\times\operatorname{Amp}_{\mathbb{R}}(X)\times\mathbb{R}$ is contractible and so $\operatorname{Stab}^{\mathrm{geo}}(X)$ is also contractible, as desired.
  \end{proof}
  
  We use Theorem \ref{theorem:mainTheoremNumbered} to show any smooth projective complex surface with finite Albanese has contractible stability manifold.
  This result generalizes \cite{fu2022}*{Theorem 1.2} from Picard rank one to arbitrary Picard rank.
  
  \begin{corollary}
    \label{corollary:finiteAlbanese}
    If $X$ is a surface with finite Albanese (i.e. the Albanese morphism $\operatorname{alb}_{X}:X\to\operatorname{Alb}(X)$ is finite onto its image) then $\operatorname{Stab}(X)$ is contractible.
  \end{corollary}
  \begin{proof}
    Since the Albanese morphism of $X$ is finite, there is a homeomorphism $\operatorname{Stab}^{\mathrm{geo}}(X)=\operatorname{Stab}(X)$~\cite{fu2022}*{Theorem 1.1}.
    Therefore, by Theorem \ref{theorem:mainTheoremNumbered}, $\operatorname{Stab}(X)$ is contractible.
  \end{proof}
  
  We note Corollary \ref{corollary:finiteAlbanese} applies to infinitely many new families of varieties .
  
  \begin{example}
    A variety $X$ has finite Albanese if and only if $X$ is a finite cover of a subvariety of an abelian variety.
   Therefore, any product of varieties with finite Albanese also has finite Albanese.
    For example, if $X$ is a finite cover of $C_{1}\times C_{2}$ for smooth positive genus curves $C_{1},C_{2}$ then $\operatorname{Stab}(X)$ is contractible.
    The Picard rank of $C_{1}\times C_{2}$ is larger than $1$ and so \cite{fu2022}*{Theorem 1.2} does not apply.
  \end{example}
  
  A recent result of Dell--Heng--Licata has also used Theorem \ref{theorem:mainTheoremNumbered} to quotients of varieties appearing in Corollary \ref{corollary:finiteAlbanese}.
  Specifically, if $X$ is a smooth projective surface with finite Albanese and $G$ is a finite group acting freely on $X$ then a distinguished component of $\operatorname{Stab}(X/G)$ is contractible~\cite{dell2024}*{Corollary 5.10}.
  
  The author is hopeful Theorem \ref{theorem:mainTheoremNumbered} generalizes to smooth projective complex threefolds where geometric stability conditions are known to exist.
  However, the techniques of this write-up do not naively extend because the proof of Theorem \ref{contractabilityMainResult} fails for threefolds.
  Namely, for threefolds geometric stability conditions are no longer determined by their central charge.
  For this reason, a new idea is necessary to extend the results of this article to threefolds.

\end{document}